\documentclass{amsart}
\usepackage{amssymb}
\usepackage{hyperref}
\usepackage{url}
\usepackage[all]{xy}
\usepackage{tikz}
\usepackage[none]{hyphenat}
\usetikzlibrary{matrix,arrows,decorations.pathmorphing}

\newtheorem{thm}{Theorem}
\newtheorem{introthm}{Theorem}

\newtheorem{lem}[thm]{Lemma}
\newtheorem{cor}[thm]{Corollary}

\newtheorem{prop}[thm]{Proposition}

\theoremstyle{definition}
\newtheorem{defn}[thm]{Definition}

\newtheorem{rem}[thm]{Remark}

\newtheorem{exa}[thm]{Example}

\newcommand{\RN}[1]{%
  \textup{\uppercase\expandafter{\romannumeral#1}}%
}

\numberwithin{thm}{section}

\newfont{\cyrr}{wncyr10}

\def\Z{\mathbf{Z}}

\def\F{\mathbf{F}}
\def\R{\mathbf{R}}
\def\C{\mathbf{C}}

\def\Ftwo{\F_2}

\def\bA{\mathbf{A}}

\def\Xset{\mathcal{C}}

\def\Xset{\mathcal{C}}

\def\O{\mathcal{O}}

\def\cP{\mathcal{P}}

\def\cP{\mathcal{P}}

\def\l{\mathfrak{q}}

\def\p{\mathfrak{p}}
\def\q{\mathfrak{q}}

\def\d{\mathfrak{d}}

\def\Hom{\mathrm{Hom}}
\def\Gal{\mathrm{Gal}}

\def\im{\mathrm{Im}}

\def\Sel{\mathrm{Sel}}

\def\Frob{\mathrm{Frob}}

\def\ur{\mathrm{ur}}

\def\res{\mathrm{res}}

\def\dimtwo{\dim_{\Ftwo}}

\def\ram{\mathrm{ram}}

\def\Pic{\mathrm{Pic}}
\def\iK{\bA_K^\times}

\def\N{\mathbf{N}}

\def\too{\longrightarrow}

\def\dirsum#1{\underset{#1}{\textstyle\bigoplus}}

\def\d2{\dim_{\Ftwo}}
\title[On 2-Selmer ranks of quadratic twists of elliptic curves]
   {On 2-Selmer ranks of quadratic twists of \\
   elliptic curves}
\author{Myungjun Yu}
\address{Department of Mathematics,
UC Irvine,
Irvine, CA 92697,
USA}
\email{\href{mailto:myungjuy@math.uci.edu}{myungjuy@math.uci.edu}}

\begin{document}

\begin{abstract}
We study the $2$-Selmer ranks of elliptic curves. We prove that for an arbitrary elliptic curve $E$ over an arbitrary number field $K$, if the set $A_E$ of 2-Selmer ranks of quadratic twists of $E$ contains an integer $c$, it contains all integers larger than $c$ and having the same parity as $c$. We also find sufficient conditions on $A_E$ such that $A_E$ is equal to $\Z_{\ge t_E}$ for some number $t_E$. When all points in $E[2]$ are rational, we give an upper bound for $t_E$. 
\end{abstract}

\maketitle
\sloppy

\section*{Introduction}
Let $E$ be an elliptic curve over a number field $K$ and let $\Sel_2(E)$ denote its $2$-Selmer group (Definition \ref{selmer}). 
Let $E^\chi$ be the quadratic twist of $E$ by a quadratic character $\chi: G_K \to \{\pm1\}$. Let
$$
r_2(E^\chi):= \dim_{\Ftwo}(\Sel_2(E^\chi)).
$$
For $E$, one may study the set
$$
A_E:= \{r_2(E^\chi): E^\chi \text{ is a quadratic twist of $E$}\}, 
$$
i.e., the set of (non-negative) integers $r$ that appear as $2$-Selmer ranks of some quadratic twists of $E$. Which integers are contained in $A_E$? and how many?

There are many interesting results in this direction. For example, Dokchitser and Dokchitser \cite{dd} showed that the elements in $A_E$ have constant parity if and only if $K$ has no real embedding and $E$ acquires everywhere good reduction over an abelian extension of $K$ (which is called ``constant $2$-Selmer parity condition" from now on). Mazur-Rubin \cite{MR} and Klagsbrun-Mazur-Rubin \cite{KMR} proved that if $E$ does not satisfy the constant $2$-Selmer parity condition and $\Gal(K(E[2])/K) \cong S_3$, then $A_E = \Z_{\ge 0}$.

When $\Gal(K(E[2])/K) \cong S_3$ or $\Z/3 \Z$, it is known that there are only three possible cases for $A_E$. The following theorem, which exhibits those possible cases, follows from inductively applying \cite[Proposition 5.2]{MR} and \cite[Proposition 6.6]{mj}.

\begin{introthm}[Mazur-Rubin, Yu]
Suppose that $\Gal(K(E[2])/K) \cong S_3$ or $\Z/3\Z$ (equivalently, $E(K)[2] = 0$). Then $A_E = \Z_{\ge 0}$, or $A_E = \{a \ge 0 : a\equiv 0 \text{ (mod $2$)} \}$, or $A_E = \{a \ge 0 : a\equiv 1 \text{ (mod $2$)}\}$.
\end{introthm}  
However, in the other cases, the behaviour of $A_E$ was less understood, so we are mainly interested in the case when $\Gal(K(E[2])/K)$ has order $1$ or $2$. Let $t_E$ denote the smallest integer in $A_E$. In this paper, we derive a result on $A_E$ by proving the following theorem (Theorem \ref{inc2}).
\begin{introthm}
\label{density1}
Let $E$ be an elliptic curve over a number field $K$. Then there exist infinitely many quadratic characters $\chi$ such that $r_2(E^\chi) = r_2(E) +2$.
\end{introthm}

By applying Theorem \ref{density1} inductively, we can see

\begin{introthm}
\label{density2}
Let $E$ be an elliptic curve over a number field $K$. Then $A_E \supset \{ r\equiv t_E \text{ (mod $2$)}: r \ge t_E\},$ (with equality if $E$ satisfies the constant $2$-Selmer parity condition). 
\end{introthm}

For an elliptic curve $E$, clearly 
\begin{equation}
\label{subset}
A_E \subset \Z_{\ge t_E}.  
\end{equation}
We find sufficient conditions on $E$ so that equality holds in \eqref{subset} (See Theorem \ref{realplace}, Theorem \ref{multiplicative} and Theorem \ref{inc2}).
\begin{introthm}
Suppose that $\Gal(K(E[2])/K)$ has order $1$ or $2$. Suppose that either 
\begin{enumerate}
\item
$K$ has a real embedding, or
\item
$\Gal(K(E[2])/K)$ has order $2$ and $E$ has multiplicative reduction at a place $\l$ such that $\l \nmid 2$ and $v_{\l}(\Delta_E)$ is odd, where $\Delta_E$ is the discriminant of a model of $E$ and $v_{\l}$ is the normalized (additive) valuation of $K_{\l}$.
\end{enumerate}
Then $A_E = \Z_{\ge t_E}.$
\end{introthm}
Let $\Sigma$ be a finite set of places of $K$ containing all primes above $2$, all primes where $E$ has bad reduction, and all infinite places. We suppose the elements (finite places) of $\Sigma$ generate the ideal class group of $K$. For $t_E$, we have a trivial lower bound $\dim_{\Ftwo}(E(K)[2])$.
However, this lower bound turns out not to be sharp in some cases. Klagsbrun \cite{klagsbrun} found examples of elliptic curves $E$ such that $t_E$ is at least $s_2+1$, where $s_2$ denotes the number of complex places of $K$ (see Example \ref{example} and Remark \ref{extendklagsbrunexample} for a discussion of this). In Section 5, when $E[2] \subset E(K)$, we give an upper bound for $t_E$ as follows (see Theorem \ref{upperbound} and Theorem \ref{upperbound2}). 
\begin{introthm}
Suppose that $E[2] \subset E(K)$. We have
$t_E \le |\Sigma| + 1$. If moreover, $E$ does not satisfy the constant $2$-Selmer parity condition, then $t_E \le |\Sigma|.$
\end{introthm}

%%%%%%%%%%%%%%%%%%%%%%%%%%%%%%%%%%%%%%%%%%%%%%%%%%%%%%%%%%%%%%%%%%%%%%%%%%

\section{preliminaries}
Let $K$ be a number field and $v$ be a place of $K$. We denote the completion of $K$ at $v$ by $K_v$. Let $E$ be an elliptic curve defined over $K$. We write $E^\chi$ for the quadratic twist of $E$ by a quadratic character $\chi$. For $d \in K^\times/(K^\times)^2$, we sometimes write $E^d$ for $E^\chi$ when $K(\sqrt{d})$ is the corresponding quadratic extension to $\chi$. For any quadratic twists $E^\chi$ of $E$, note that there is a canonical ($G_K$-module) isomorphism $E[2] \cong E^\chi[2]$. Throughout the paper, for (topological) groups A and B, we denote the group of continuous homomorphisms from $A$ to $B$ simply by $\Hom(A, B)$.

\begin{defn}
Let $L$ be a field of characteristic $0$. We write
$$
\Xset(L) := \Hom(G_L, \{\pm1\}).
$$
If $L$ is a local field, we often identify $\Xset(L)$ with $\Hom(L^\times, \{\pm1\})$ via the local reciprocity map, and let $\Xset_\ram(L) \subset \Xset(L)$ be the subset of ramified characters in $\Xset(L)$ ($\chi \in \Xset_\ram(L)$ if and only if $\chi(\O_L^\times) \neq 1$, where $\O_L^\times$ is the unit group of the ring of integers $\O_L$ of $L$, by local class field theory). 
\end{defn}

\begin{defn}
\label{localcondition}
For $\chi \in \Xset(K_v)$, define
$$
\alpha_v(\chi) := \mathrm{Im}( E^{\chi}(K_v)/2 E^{\chi}(K_v) \to H^1(K_v, E^{\chi}[2]) \cong H^1(K_v, E[2])),
$$
where the first map is given by the Kummer map. Define
$$
h_v(\chi):= \dimtwo(\alpha_v(1_v)/(\alpha_v(1_v) \cap \alpha_v(\chi))).
$$
\end{defn}

\begin{lem}
\label{hvkramer}
For $\chi \in \Xset(K_v)$, let $L = \overline{K_v}^{\ker(\chi)}$. Then
$$
h_v(\chi) = \dimtwo(E(K_v)/\N E(L)),
$$
where $\N E(L)$ is the image of the norm map $\N: E(L) \to E(K_v)$. 
\end{lem}
\begin{proof}
This is \cite[Proposition 7]{kramer}. 
\end{proof}

\begin{thm}
\label{tlp}
The Tate local duality and the Weil pairing give a nondegenerate pairing
\begin{equation}
\label{tld}
\langle  \text{ , } \rangle_v : H^1(K_v, E[2]) \times H^1(K_v, E[2]) \too H^2(K_v, \{\pm1\}),
\end{equation}
where $H^2(K_v, \{\pm1\}) \cong \Ftwo$ unless $v$ is a complex place.
\end{thm}

\begin{proof}
For example, see \cite[Theorem 7.2.6]{cohomology}.
\end{proof}

%%%%%%%%%%%%%%%%%%%%%%%%%%%%%%%%%%%%%%%%%%%%%%%%%%%%%%%%%%%%%%%%%%%%%%%%%%%%%%%%%

\section{Selmer groups and comparing local conditions}

Let $K$ be a number field and $E$ be an elliptic curve defined over $K$. We fix embeddings $\overline{K} 
\hookrightarrow \overline{K_v}$ for all places $v$ so that $G_{K_v} \subset G_K$.

\begin{defn}
\label{restrictionmap}
For every place $v$ of $K$, we let
$$
\res_v : H^1(K, E[2]) \to H^1(K_v, E[2])
$$
denote the restriction map of group cohomology. Let $T$ be a finite set of places. let
$$
\res_T :  H^1(K, E[2]) \to \bigoplus_{v \in T} H^1(K_v, E[2])
$$
denote the sum of restriction maps. 
\end{defn}

\begin{defn}
\label{selmer}
Let $\chi \in \Xset(K)$. The $2$-Selmer group $\Sel_2(E^\chi) \subset H^1(K,E[2])$ is the (finite)
$\F_2$-vector space defined by the following exact sequence
$$
0 \too \Sel_2(E^\chi) \too H^1(K,E[2]) \too \dirsum{v} H^1(K_v,E[2])/\alpha_v(\chi_v),
$$
where the rightmost map is the sum of the restriction maps, and $\chi_v$ is the restriction of $\chi$ to $G_{K_v}$. In particular, if $\chi$ is the trivial character, it is the classical $2$-Selmer group of $E$. 
\end{defn}
We define various Selmer groups as follows.

\begin{defn}
\label{variousselmer}
Let $T$ be a finite set of places of $K$. Let $S = \{v_1, \cdots, v_k\}$ be a (finite) set of places such that $S \cap T = \varnothing$. Let $\psi_{v_j} \in \Xset(K_{v_j})$.  
Define
\begin{align*}
\Sel_{2}(E, \psi_{v_1}, \cdots, \psi_{v_k}) := \{x \in H^1(K, E[2])|& \mathrm{res}_{v}(x) \in \alpha_v(1_v) \text{ if } v \notin S, \text{ and }\\
 &\mathrm{res}_{v_j}(x) \in \alpha_{v_j}(\psi_{v_j}) \text{ for $1 \le j \le k$} \}.
\end{align*}
Define
$$
\Sel_{2, T}(E, \psi_{v_1}, \cdots, \psi_{v_k}) := \{x \in \Sel_{2}(E, \psi_{v_1}, \cdots, \psi_{v_k})|\mathrm{res}_{T}(x) = 0  \}.
$$
Define
\begin{align*}
\Sel_{2}^T(E, \psi_{v_1}, \cdots, \psi_{v_k}) := \{x \in H^1(K, E[2])|& \mathrm{res}_{v}(x) \in \alpha_v(1_v) \text{ if } v \notin S \cup T, \text{ and }\\
 &\mathrm{res}_{v_j}(x) \in \alpha_{v_j}(\psi_{v_j}) \text{ for $1 \le j \le k$} \}.
\end{align*}
For a place $v\notin S$, we simply write $\Sel_{2, v}(E, \psi_{v_1}, \cdots, \psi_{v_k}), \text{ } \Sel_2^v(E, \psi_{v_1}, \cdots, \psi_{v_k})$ for $\Sel_{2, \{v\}}(E, \psi_{v_1}, \cdots, \psi_{v_k})$, $\Sel_2^{\{v\}}(E, \psi_{v_1}, \cdots, \psi_{v_k})$, respectively. 
\end{defn}

\begin{defn}
For convenience, we write $r_2(E^\chi), r_2(E, \psi_{v_1},\cdots, \psi_{v_n})$ for $\dimtwo(\Sel_2(E^\chi)), \dimtwo(\Sel_{2}(E, \psi_{v_1},\cdots, \psi_{v_n}))$, respectively. 
\end{defn}

The following theorem is due to \cite[Theorem 3.9 and Lemma 5.2(ii)]{KMR}.

\begin{thm}[Kramer, Klagsbrun-Mazur-Rubin]
\label{parity}
Let $\chi \in \Xset(K)$. We have
$$
 r_2(E) -  r_2(E^\chi) \equiv \displaystyle{\sum_{v}}h_v(\chi_v) (\textrm{mod } 2),
$$
where $\chi_v$ is the restriction of $\chi$ to $G_{K_v}$ and $h_v$ is given in Definition \ref{localcondition}. Let $S = \{v_1, \cdots, v_k\}$ be a (finite) set of places. Let $\psi_{v_i} \in \Xset(K_{v_i})$. We have
$$
r_2(E, \psi_{v_1}, \cdots, \psi_{v_k}) - r_2(E) \equiv \displaystyle{\sum_{i=1}^k}h_{v_i}(\psi_{v_i}) (\textrm{mod } 2).
$$
\end{thm}

\begin{lem}
 \label{localzero}
Let $\chi \in \Xset(K_v)$. Suppose that $\chi$ satisfies one of the following conditions:
\begin{itemize}
\item
$\chi$ is trivial, or
\item
$E/K_v$ has good reduction, $v \nmid \infty$, and $\chi$ is unramified.
\end{itemize}
Then $h_v(\chi) = 0$, i.e., $\alpha_v(1_v) = \alpha_v(\chi)$.
 \end{lem}
\begin{proof}
Let $L = \overline{K}_v^{\ker(\chi)}$. In either case, $\N(E(L)) = E(K_v)$ (in the second case, it follows from \cite[Corollary 4.4]{norm}). Thus, by Lemma \ref{hvkramer}, the result follows.
\end{proof}

From now on, let $\Sigma$ denote a finite set of places of $K$ containing all primes above $2$, all primes where $E$ has bad reduction, and all infinite places.

\begin{defn}
\label{cpi}
Define
\begin{align*}
\cP_{i} &:=\; \{\l : \text{$\l \notin \Sigma$
   and $\dim_{\Ftwo}(E(K_\l)[2]) = i$}\}\quad \text{for $0 \le i \le 2$}, \text{ and} \\
\cP &:=\; \cP_0 \textstyle\coprod \cP_1 \coprod \cP_2  = \{\l : \l \notin \Sigma \}.
\end{align*}
\end{defn}
Although $\cP_i$ and $\cP$ depend on the choice of $\Sigma$ and $E$, we suppress them from the notation. 
\begin{rem}
\label{localdim}
By \cite[Lemma 2.11(i)]{mj}, if $\l \in \cP$, we have 
$$\dimtwo(E(K_\l)/2E(K_\l)) = \dimtwo(E(K_\l)[2]).$$ Hence, if $\l \in \cP_i$ and $\chi \in \Xset(K_\l)$, we have $\dimtwo(\alpha_\l(\chi)) = i$.     
\end{rem}

\begin{lem}
\label{ramhv}
Let $\l \in \cP_i$. Suppose that $\chi \in \Xset_{\ram}(K_\l)$. Then $\alpha_\l(1_\l) \cap \alpha_\l(\chi) = \{0 \}$,
and $h_\l(\chi) = i$.
\end{lem}

\begin{proof}
See \cite[Lemma 2.11]{MR}. 
\end{proof}

\begin{thm}
Let $T$ be a finite set of places of $K$. Let $v_1, \cdots, v_k \not\in T$ be places and $\psi_{v_j} \in \Xset(K_{v_j})$.
\label{ptd}
The images of right hand restriction maps of the following exact sequences are orthogonal complements with respect to the pairing given by the sum of pairings \eqref{tld} of the places $v\in T$
$$
\xymatrix@R=3pt@C=7pt{
0 \ar[r] & \Sel_2(E, \psi_{v_1}, \cdots, \psi_{v_k}) \ar[r] & \Sel_2^T(E, \psi_{v_1}, \cdots, \psi_{v_k}) \ar[r]
    & \bigoplus_{v \in T} H^1(K_v,E[2])/\alpha_v(1_v),  \\
0 \ar[r] & \Sel_{2, T}(E, \psi_{v_1}, \cdots, \psi_{v_k}) \ar[r] & \Sel_2(E, \psi_{v_1}, \cdots, \psi_{v_k}) \ar[r] & \bigoplus_{v \in T} \alpha_v(1_v).
}
$$
In particular, 
\begin{align*}
\dimtwo(\Sel_2^T(E, \psi_{v_1}, \cdots,& \psi_{v_k})) - \dimtwo(\Sel_{2, T}(E, \psi_{v_1}, \cdots, \psi_{v_k}))\\ & = \Sigma_{v \in T} \dimtwo(\alpha_v(1_v))
= \Sigma_{v \in T} \frac{1}{2}\dimtwo(H^1(K_v, E[2])).
\end{align*}
\end{thm}

\begin{proof}
The lemma follows from the Global Poitou-Tate Duality. For example, see \cite[Theorem 2.3.4]{kolyvagin}.
\end{proof}

\begin{cor}
\label{ressur}
Suppose $T= \{\l_1, \cdots, \l_n \}$, where  $ \l_i \in \cP$. Let $\psi_i \in \Xset_{\ram}(K_{\l_i})$. Let $v_0 \not\in T$ be a place and $\psi_{v_0} \in \Xset(K_{v_0})$. Suppose that the map $\res_T : \Sel_2(E, \psi_{v_0}) \to \bigoplus_{v \in T} \alpha_v(1_v)$ is surjective. Then we have 
\begin{enumerate}
\item
$\Sel_2(E, \psi_{v_0})  = \Sel_2^T(E, \psi_{v_0}), \text{ and }$
\item
$\Sel_2(E, \psi_1, \cdots, \psi_n, \psi_{v_0}) = \Sel_{2, T}(E, \psi_{v_0}).$
\end{enumerate}
\end{cor}

\begin{proof}
The first assertion is clear because the orthogonality in Theorem \ref{ptd}  shows that the image of 
$$
\res_T : \Sel_2^T(E, \psi_{v_0}) \to \bigoplus_{v \in T} H^1(K_v,E[2])/\alpha_v(1_v)
$$
is trivial. Lemma \ref{ramhv} shows that
$$
\Sel_2(E, \psi_{v_0}) \cap \Sel_2(E, \psi_1, \cdots, \psi_n, \psi_{v_0}) = \Sel_{2, T}(E, \psi_{v_0}), 
$$
where the intersection is taken in $H^1(K, E[2])$. Now the second assertion is easy to see.
\end{proof}

\begin{cor}
\label{babo}
Let $\l$ be a place and let $v_1, \cdots, v_k$ be places of $K$ not equal to $\l$. Let $\psi_{v_j} \in \Xset(K_{v_j})$. For any $\phi_{\l}, \eta_{\l} \in \Xset(K_\l)$, we have $$|r_2(E, \psi_{v_1}, \cdots, \psi_{v_k}, \phi_{\l}) - r_2(E, \psi_{v_1}, \cdots, \psi_{v_k}, \eta_{\l})| \le \dimtwo(\alpha_\l(1_\l)).$$
\end{cor}

\begin{proof}
In Theorem \ref{ptd}, take $T = \{\l\}.$ Note that $\Sel_2(\psi_{v_1}, \cdots, \psi_{v_k}, \phi_{\l})$ and  $\Sel_2(\psi_{v_1}, \cdots, \psi_{v_k}, \eta_{\l})$ contains $\Sel_{2, \l}(\psi_{v_1}, \cdots, \psi_{v_k}) $ and are contained in $\Sel_{2}^\l(\psi_{v_1}, \cdots, \psi_{v_k})$, where the result easily follows from Theorem \ref{ptd}.
\end{proof}

\section{Increasing $2$-Selmer rank by twisting}
Let $E$ be an elliptic curve over a number field $K$ and let $\Sigma$ be as in previous section. 
\begin{lem}
\label{4torsion}
Let $\l$ be a prime of $K$ such that $\l \nmid 2$. Then
\begin{enumerate}
\item
if all the points of $E[4]$ are $K_\l$-rational and $\chi$ is a nontrivial quadratic character, then
$E^\chi(K_\l)[4] = E^\chi(K_\l)[2] \cong (\Z/2\Z)^2$;
\item
if $E(K_\l)[4] = E(K_\l)[2]$, then the map $E(K_\l)[2] \to E(K_\l)/2E(K_\l)$ via the projection is an isomorphism. 
\end{enumerate}
\end{lem}

\begin{proof}
The first assertion (i) is obvious from the definition of quadratic twists. For (ii), multiplication by $2$ is surjective on the pro-(prime to $2$) part of $E(K_{\l})$, so only the pro-$2$ part $E(K_\l)[2^\infty]$ contributes to $E(K_{\l})/2E(K_{\l})$, hence $$E(K_\l)[2] =  E(K_{\l})[2^\infty]/2E(K_{\l})[2^\infty] \cong E(K_{\l})/2E(K_{\l}). $$  
\end{proof}

The following generalizes methods that are used in the proof of Proposition 5.1 in \cite{MR}.
\begin{thm}
\label{inc2}
Let $E$ be an elliptic curve over a number field $K$. Then there exist infinitely many $\chi \in \Xset(K)$ such that $r_2(E^\chi) = r_2(E) +2$.
\end{thm}

\begin{proof}
If $\Gal(K(E[2])/K) \cong S_3$ or $A_3$, the result follows from \cite[Proposition 6.6]{mj}. Therefore, from now on, we assume that  $\Gal(K(E[2])/K)$ has order $1$ or $2$, i.e., there exists a non-trivial rational $2$-torsion  point $P \in E(K)[2]$. Let $\theta$ be the formal product of $8$, and all places in $\Sigma$ not dividing $2$. In particular, $\theta$ is divisible by primes where $E$ has bad reduction. Let $K[\theta]$ be the maximal $2$-subextension of $K(\theta)$, where $K(\theta)$ is the ray class field  modulo $\theta$. 

Let $L$ be a Galois extension containing $K(E[4])K[\theta]$ such that the image of the restriction map
$$
\Sel_2(E) \subseteq H^1(K, E[2]) \to H^1(L, E[2]) = \Hom(G_L, E[2])
$$
is trivial. Choose a prime (Chebotarev's density theorem) $\l \notin \Sigma$ so that $\l$ is unramified in $L/K$ and $\Frob_\l|_L = 1$. Note that the restriction map $H^1(K, E[2]) \to H^1(K_\l, E[2])$ factors through the restriction $H^1(K,E[2]) \to H^1(L, E[2])$ because $\l$ splits completely in $L/K$, so $\res_{\l}(\Sel_2(E)) = 0$ and  
$$
\Sel_2(E) = \Sel_{2, \l}(E).
$$ 
Moreover, there exists an odd integer $k$ such that $\l^k = (d)$ for some $d \in K^\times$ such that $d \equiv 1 \text{ } (\text{mod } \theta)$. Note the following properties of the extension $K(\sqrt{d})/K$:
\begin{itemize}
\item
$\q$ is ramified in $K(\sqrt{d})/K$,
\item
If $v \notin \Sigma$ and $v \neq \l$, then $v$ is unramified in $K(\sqrt{d})/K$, and
\item
If $v \in \Sigma$, then $v$ splits in $K(\sqrt{d})/K$.
\end{itemize}
Therefore by Lemma \ref{localzero}, the local conditions of $\Sel_2(E)$ and $\Sel_2(E^d)$ are the same except at $\l$, where two local conditions intersect trivially by Lemma \ref{ramhv}. By Corollary \ref{babo} and the fact that $\Sel_2(E) = \Sel_{2, \l}(E)$, we have $0 \le r_2(E^d) - r_2(E) \le 2$. Moreover since $\l \in \cP_2$, Theorem \ref{parity} and Lemma \ref{ramhv} prove that 
\begin{equation}
\label{twotwo}
r_2(E^d) = r_2(E) \text{ or } r_2(E) + 2.
\end{equation}
By our choice of a prime $\l$, we have $E[4] \subset E(K_\l)$. By Lemma \ref{4torsion}, $P$ has a nonzero local Kummer image for $E^d$ at $\l$. Therefore $\res_{\l}(\Sel_2(E^d)) \neq 0$, where $\res_{\l} : \Sel_2(E^d) \to H^1(K_\l, E[2])$ is the restriction map.
Hence $\Sel(E^d)$ contains $\Sel_2(E) (= \Sel_{2, \l}(E))$ properly, i.e., $r_2(E^d) \geq r_2(E) +1$. Therefore by \eqref{twotwo}, we have $r_2(E^d)= r_2(E) +2$. Since the only constraint on our choice of $\l$ is $\Frob_\l|_L = 1$ and there are infinitely many such primes (Chebotarev's density theorem), we have infinitely many quadratic twists with the desired property. 
\end{proof}

\begin{rem}
A similar argument can show the following theorem: Let $C_f$ be a hyperelliptic curve over a number field $K$ given by an affine model
$$
y^2 = f(x),
$$
where $n:=\deg(f) > 1$ is odd. Let $J$ be the Jacobian of $C_f$. If $K$ contains a root of $f$, then for any given natural number $r$, there exist infinitely many quadratic twists $J^\chi$ such that $\dimtwo(\Sel_2(J^\chi/K)) \ge r$ (In \cite{mj}, the author discusses the cases when $\Gal(f) \cong A_n$ or $S_n$. In such cases, the result is even stronger. See \cite[Theorem 6.7]{mj}). 
\end{rem}

\section{Changing the parity of $2$-Selmer rank by twisting}
Recall that $\Sigma$ is a finite set of places of $K$ containing all places where $E$ has bad reduction, all primes above $2$, and all infinite places. We enlarge $\Sigma$, if necessary, so that $\Pic(O_{K, \Sigma}) = 1$, where $O_{K, \Sigma}$ denote the ring of $\Sigma$-integers. For the rest of the paper, we put $n:=|\Sigma|$. Let $\Delta_E$ denote the discriminant of some model of the elliptic curve $E$.

\begin{lem}
\label{unitdim}
$\dimtwo(O_{K, \Sigma}^\times/(O_{K, \Sigma}^\times)^2) = n.$
\end{lem}

\begin{proof}
It is well-known that $O_{K, \Sigma}^\times \cong \Z^{n -1} \oplus \Z/m\Z$, where $m = \#\{\text{roots of unity in $K$}\}$ is divisible by $2$ (for example, see \cite[Proposition 6.1.1]{neukirch}).
\end{proof}

\begin{lem}
\label{findglobalchar}
Let $\l \notin \Sigma$ (so $\l \nmid 2$) be a prime of $K$ and suppose $g \in \Hom(\O_\l^\times, \{\pm1\})$ is non-trivial. Then $g(b) = \Frob_\l(\sqrt{b})/\sqrt{b}$ for all $b \in \O_{K, \Sigma}^\times$. In particular, if $\psi \in \Xset_{\ram}(K_\l)$, then $\psi(b) = \Frob_\l(\sqrt{b})/\sqrt{b}$ for all $b \in \O_{K, \Sigma}^\times$.
\end{lem}

\begin{proof}
We have
$$
\Hom(\O_\l^\times, \{\pm1\}) = \Hom(\O_\l^\times/(\O_\l^\times)^2, \{\pm1\}) \cong \Z/2\Z
$$ 
because $\O_\l^\times/(\O_\l^\times)^2 \cong \Z/2\Z.$ Note that $b \in (\O_\l^\times)^2$ if and only if $\Frob_\l(\sqrt{b}) = \sqrt{b}$, where the assertion follows. 
\end{proof}

\begin{lem}
\label{cft}
The image of the restriction map
\begin{align*}
\Xset(K) = \Hom(\iK/K^\times,\{\pm1\}) &= \textstyle\Hom((\prod_{\mu \in \Sigma}K_{\mu}^\times \times \prod_{\nu\notin\Sigma}\O_{\nu}^\times)/\O_{K,\Sigma}^\times,\{\pm1\}) \\
& \too \textstyle{\prod_{\mu \in \Sigma}\Hom(K_{\mu}^\times, \{\pm1\})\times \prod_{\nu \not\in \Sigma}\Hom(\O_{\nu}^\times, \{\pm1\})}
\end{align*}
is the set of all $(((f_{\mu})_{\mu \in \Sigma}), ((g_{\nu})_{\nu \not\in \Sigma}))$ such that $\prod_{\mu \in \Sigma} f_{\mu}(b) \prod_{\nu \not\in \Sigma} g_{\nu}(b) = 1$ for all $b \in \O_{K,\Sigma}^\times$, where $f_{\mu} \in \Hom(K_{\mu}^\times, \{\pm1\})$ and $g_{\nu}\in \Hom(\O_{\nu}^\times, \{\pm1\})$.
\end{lem}

\begin{proof}
Global Class Field Theory and the condition $\Pic(\O_{K, \Sigma}) = 1$ show the equalities. It is clear that the image is as stated. 
\end{proof}

\begin{prop}
\label{lastcase}
Let $v_0 \in \Sigma$ and $\psi_{v_0} \in \Xset(K_v)$. Suppose that $\psi_{v_0}(\O_{K, \Sigma}^\times) = 1$. Then there exists $\chi \in \Xset(K)$ such that $\Sel_2(E^\chi) = \Sel_2(E, \psi_{v_0})$
\end{prop}

\begin{proof}
Put $f_\mu \in \Hom(K_\mu^\times, \{\pm1\})$ for $\mu \in \Sigma$ and $g_\nu \in \Hom(\O_\nu^\times, \{\pm1\})$ for $ \nu \not\in \Sigma$ such that
\begin{itemize}
\item
$f_{v_0} = \psi_{v_0}$,
\item
$f_v = 1_v$ for $v \in \Sigma \backslash \{v_0\}$, and
\item
$g_{\p}$ is trivial for $\p \notin \Sigma$.
\end{itemize} 
By Lemma \ref{cft}, there exists a character $\chi \in \Xset(K)$ such that for $\mu \in \Sigma$ and $\nu \notin \Sigma$, $\chi_\mu = f_\mu$ and $\chi_\nu|_{\O_\nu^\times} = g_\nu$, where $\chi_\mu, \chi_\nu$ are restrictions of $\chi$ to $K_{\mu}^\times, K_{\nu}^\times$ via the local reciprocity maps, respectively. Now one can see the local conditions for $\Sel_2(E^\chi)$ and $\Sel_2(E, \psi_{v_0})$ are the same everywhere by Lemma \ref{localzero}. 
\end{proof}

\begin{lem}
\label{idea}
Let $v_0$ be a place in $\Sigma$ and let $T$ be a (finite) set of primes such that $T \cap \Sigma = \varnothing$. Suppose that $\psi_{v_0} \in \Xset(K_{v_0})$. Then there exist infinitely many primes $\l \notin \Sigma \cup T$ for which there exists a character $\chi \in \Xset(K)$ satisfying the following conditions.
\begin{enumerate}
\item
$\chi_{v_0} = \psi_{v_0}$,
\item
$\chi_{v} = 1_v$ for $v \in \Sigma \backslash \{v_0\}$,
\item
$\chi_{\omega}$ is ramified for $\omega \in T$,
\item
$\chi_{\p}$ is unramified for $\p \notin \Sigma \cup T \cup \{\l\}$,
\item
$\chi_\l$ is ramified, 
\end{enumerate}
where $\chi_{v_0}, \chi_{v}, \chi_{\omega}, \chi_{\p}, \chi_\l$ are restrictions of $\chi$ to $K_{v_0}^\times, K_{v}^\times, K_{\omega}^\times , K_{\p}^\times, K_{\l}^\times$ via the local reciprocity maps, respectively. 
\end{lem}

\begin{proof}
Let $\beta_1, \cdots, \beta_n$ be a basis of $O_{K, \Sigma}^\times/(O_{K, \Sigma}^\times)^2$. Choose a prime $\l$ such that 
\begin{equation}
\label{choiceofq}
\Frob_\l(\sqrt{\beta_i})/\sqrt{\beta_i} = \psi_{v_0}(\beta_i) \cdot \prod_{\omega \in T} \Frob_\omega(\sqrt{\beta_i})/\sqrt{\beta_i}
\end{equation}
for all $i$, where the existence is guaranteed by Chebotarev's density theorem. Put $f_\mu \in \Hom(K_\mu^\times, \{\pm1\})$ for $\mu \in \Sigma$ and $g_\nu \in \Hom(\O_\nu^\times, \{\pm1\})$ for $ \nu \not\in \Sigma$ such that
\begin{itemize}
\item
$f_{v_0} = \psi_{v_0}$,
\item
$f_v = 1_v$ for $v \in \Sigma \backslash \{v_0\}$,
\item
$g_{\omega}$ is not trivial for $\omega \in T$,
\item
$g_{\p}$ is trivial for $\p \notin \Sigma \cup T \cup \{\l\}$, and
\item
$g_\l$ is not trivial.
\end{itemize}
By Lemma \ref{findglobalchar}, we have $$
g_\l(\beta_i) = f_{v_0}(\beta_i) \cdot \displaystyle{\prod_{v \in \Sigma \backslash \{v_0\}} f_v(\beta_i) \cdot \prod_{\omega \in T} g_\omega(\beta_i) \cdot \prod_{\p \notin \Sigma \cup \{\l\} \cup T} g_\p(\beta_i)}.
$$
By Lemma \ref{cft}, this means that there exists a character $\chi \in \Xset(K)$ such that for $\mu \in \Sigma$ and $\nu \notin \Sigma$, $\chi_\mu = f_\mu$ and $\chi_\nu|_{\O_\nu^\times} = g_\nu$, where $\chi_\mu, \chi_\nu$ are restrictions of $\chi$ to $K_{\mu}^\times, K_{\nu}^\times$ via the local reciprocity maps, respectively. It is easy to see $\chi$ satisfies the desired conditions. For example, for $\omega \in T$, $\chi_\omega|_{\O_\omega^\times} = g_\omega$, and this shows that $\chi_\omega$ is ramified since $g_\omega(\O_\omega^\times) \neq 1$ by our construction.  
\end{proof}

\begin{prop}
\label{idea2}
Let $v_0 \in \Sigma$ and $\psi_{v_0} \in \Xset(K_{v_0})$.
\begin{enumerate}
\item
If $\psi_{v_0}(\Delta_E) = -1$ and $\Gal(K(E[2])/K) \cong \Z/2\Z$, there exist infinitely many $\varphi \in \Xset(K)$ such that $r_2(E^\varphi/K) = r_2(E, \psi_{v_0}) +1$.
\item
If $\psi_{v_0}(\Delta_E) = 1,$ there exist infinitely many $\varphi \in \Xset(K)$ such that $r_2(E^\varphi/K) = r_2(E, \psi_{v_0}) + 2$.
\item
Suppose that $\psi_{v_0}(\Delta_E) = 1$ and there exists an element $c \in \Sel_2(E, \psi_{v_0})$. Let $T = \varnothing$ and choose $\l$ and $\chi$ as in Lemma \ref{idea}. Suppose that $\res_\l(c) \neq 0$. Then there exist infinitely many $\varphi \in \Xset(K)$ such that $r_2(E^\varphi/K) = r_2(E, \psi_{v_0})$.
\end{enumerate}
\end{prop}

\begin{proof}
For (i) and (ii), let $T = \varnothing$ and we begin with choosing $\l$ and $\chi$ as in Lemma \ref{idea}. Note that the local conditions for $\Sel_2(E^\chi)$ and $\Sel_2(E, \chi_{v_0})$ are the same everywhere except possibly at $\l$ by Lemma \ref{localzero}. Thus Corollary \ref{babo} shows that $|r_2(E^\chi) - r_2(E, \chi_{v_0})| \le 2$. The conditions in (i) and the product formula imply $\chi_\l(\Delta_E) = \psi_{v_0}(\Delta_E) = -1$, so $\Delta_E \notin (K_\l^\times)^2$, which shows that $E(K_\l)[2] \cong \Z/2\Z$. Hence Theorem \ref{parity}, Lemma \ref{ramhv} prove that $r_2(E^\chi)$ is $r_2(E, \psi_{v_0}) - 1,$ or $r_2(E, \psi_{v_0})+1$. Then (i) follows from Theorem \ref{inc2}. For (ii), the condition $\psi_{v_0}(\Delta_E) = 1$ and the product formula imply $\chi_\l(\Delta_E) = 1$, so $\Delta_E \in (K_\l^\times)^2$, which shows that $E(K_\l)[2] \cong (\Z/2\Z)^2$ or $E(K_\l)[2] = 0$. Then Theorem \ref{parity}, Lemma \ref{ramhv} show that $r_2(E^\chi)$ is $r_2(E, \psi_{v_0}) - 2,$ or $r_2(E, \psi_{v_0})$ or $r_2(E, \psi_{v_0}) + 2$ and the rest follows from Theorem \ref{inc2}. To see (iii), note that the condition $\res_\l(c) \neq 0$ rules out the possibility for $r_2(E^\chi)$ to be $r_2(E, \psi_{v_0}) + 2$ in the proof of (ii) (for otherwise, $r_2(E^\chi) \ge \dimtwo(\Sel_{2, \l}(E^\chi)) + 3$ and this would mean $r_2(E^\chi) \ge \dimtwo(\Sel_{2}^{\l}(E^\chi)) + 1$, which is absurd).  
\end{proof}

\begin{lem}
\label{real}
Suppose that $K$ has a real place $v_0$, so $K_{v_0} \cong \R$. Let $\eta \in \Xset(K_{v_0})$ be the sign character. Then 
\[
h_{v_0}(\eta) =
\begin{cases}
0 & \text{if $\dimtwo(E(K_{v_0})[2]) = 1$}, \\
1 & \text{if $\dimtwo(E(K_{v_0})[2]) = 2$}.
\end{cases}
\]
\end{lem}
\begin{proof}
The image $\N(E(\C))$ of the norm map
$$
\N: E(\C) \to E(\R)
$$
is the connected component of the identity of $E(\R)$, i.e., $\N(E(\C)) \cong \R/\Z$, where the result follows by Lemma \ref{hvkramer}.
\end{proof}

\begin{lem}
\label{resinjlem}
Let $M= K(E[2])$. The restriction map
\begin{equation}
\label{resinj}
H^1(K,E[2]) \to H^1(M, E[2]) = \Hom(G_M, E[2])
\end{equation}
is an injection. 
\end{lem}

\begin{proof}
The Inflation-Restriction Sequence shows that the kernel of \eqref{resinj} is $H^1(M/K, E[2])$. It is well-known that $H^1(\mathrm{GL}_2(\Z/2\Z), E[2]) = 1$ ($\mathrm{GL}_2(\Z/2\Z) \cong S_3$). For the rest cases, let $\sigma$ be a generator of the cyclic group $\Gal(M/K)$. One can see $\mathrm{Ker}(\sigma +1) = \mathrm{Im}(\sigma -1)$, so the cohomology group vanishes. 
\end{proof}

\begin{thm}
\label{realplace}
If $K$ has a real embedding, there exist infinitely many $\chi \in \Xset(K)$ such that $r_2(E^\chi) = r_2(E) + 1$. 
\end{thm}

\begin{proof}
Let $M = K(E[2])$. We assume $\Gal(M/K)$ has order $1$ or $2$, since otherwise we already know the result holds by 
\cite[Theorem 1.5]{MR} and Theorem \ref{inc2}. We let $v_0$ be a real place, so that $K_{v_0} \cong \R$. Let $\psi_{v_0} \in \Xset(K_{v_0})$ denote the sign character, i.e., $\psi_{v_0}$ sends negative numbers to $-1$.

 Case 1: $E[2] \subset E(K)$. We have $E(K_{v_0}) \cong \R/\Z \oplus \Z/2\Z$. Therefore, there exists a point $P \in E(K)[2]$ that is not divisible by $2$ in $E(K_{v_0})$. One can see $\res_{v_0}(\overline{P}) \neq 0$, where $\overline{P}$ is the image of $P$ in the map $E(K) \to E(K)/2E(K) \to \Sel_2(E) \subset H^1(K, E[2])$, because the image of $P$ in $E(K_{v_0})/2E(K_{v_0})$ is not trivial. The restriction map $\Sel_2(E)/\Sel_{2,v_0}(E) \to \alpha_{v_0}(1_{v_0})$ is an isomorphism (since $\res_{v_0}(\Sel_2(E)) \neq 0$) and the restriction map $\Sel_2(E, \psi_{v_0})/\Sel_{2,v_0}(E) \to \alpha_{v_0}(\psi_{v_0})$ is an injection. 
Therefore, Theorem \ref{parity} and Lemma \ref{real} show that $r_2(E, \psi_{v_0}) = r_2(E) - 1$. Then the result follows from Proposition \ref{idea2}(ii). 

 Case 2: $\Gal(M/K) \cong \Z/2\Z$ and $E(K_{v_0}) \cong \R/\Z$ (i.e., $\psi_{v_0}(\Delta_E) = -1$). We have $\Sel_2(E) = \Sel_2(E, \psi_{v_0})$ since $\alpha_{v_0}(1_v),\alpha_{v_0}(\psi_{v_0}) \subset H^1(\R, E[2]) = 0$ in this case. The result follows form Proposition \ref{idea2}(i). 

 Case 3: $\Gal(M/K) \cong \Z/2\Z$ and $E(K_{v_0}) \cong \R/\Z \oplus \Z/2\Z$ ($\Delta_E \not\in (K^\times)^2$ and $\psi_{v_0}(\Delta_E) = 1$). Suppose that $\beta_1, \cdots, \beta_{n-1}, \Delta_E$ form a basis of $\O_{K, \Sigma}^\times/(\O_{K, \Sigma}^\times)^2$. By Corollary \ref{babo}, we have $|r_2(E, \psi_{v_0}) - r_2(E)| \le 1$. Then $r_2(E, \psi_{v_0}) = r_2(E) + 1$  or $r_2(E) - 1$ by Theorem \ref{parity} and Lemma \ref{real}. If  $r_2(E, \psi_{v_0}) = r_2(E) - 1$, Proposition \ref{idea2}(ii) proves the result. Hence for the rest of the proof, we assume $r_2(E, \psi_{v_0}) = r_2(E) + 1$. Choose $c \in \Sel_2(E, \psi_{v_0}) \backslash \Sel_2(E)$. Then $\res_{v_0}(c) \neq 0$. Let $\tilde{c}$ denote the image of $c$ in the map \eqref{resinj} in Lemma \ref{resinjlem}. Let $L:= M(\sqrt{\beta_1}, \cdots, \sqrt{\beta_{n-1}})$ and $N:=\overline{M}^{\ker(\tilde{c})}$ (we identify $\overline{K}$ and $\overline{M}$).

(i) First, suppose that $N \not\subset L$. 
 Choose $\l \in \cP_2$ so that
\begin{itemize}
\item
$\l$ is unramified in $NL/M$
\item
$\Frob_\l(\sqrt{\beta_i})/\sqrt{\beta_i} = \psi_{v_0}(\beta_i),$
\item
$\Frob_\l|_{\Gal(N/M)} \neq 1$, i.e., $N \not\subset K_\l$. 
\end{itemize}
It is possible because $N \not\subset L$. Note that $\l$ is chosen as in Lemma \ref{idea} for $T = \varnothing$ (see \eqref{choiceofq}). Then $\res_\l(c) \neq 0$ since $N \not\subset K_\l$. The result follows from Proposition \ref{idea2}(iii).

(ii) Now we assume that $N \subset L$. By choosing a basis again, we may assume that $\psi_{v_0}(\beta_1) = -1$ and $\psi_{v_0}(\beta_2) = \psi_{v_0}(\beta_3) = \cdots = \psi_{v_0}(\beta_{n-1}) = \psi_{v_0}(\Delta_E) = 1$. Since $\res_{v_0}(c) \neq 0$, we have $N \not\subset M(\sqrt{\beta_2}, \sqrt{\beta_3}, \cdots ,\sqrt{\beta_{n-1}}) (= L \cap K_{v_0})$. Choose $\l \in \cP_2$ so that
\begin{itemize}
\item
$\l$ is unramified in $L/K$
\item
$\Frob_\l(\sqrt{\beta_i})/\sqrt{\beta_i} = \psi_{v_0}(\beta_i).$
\end{itemize}
Clearly, $L \cap K_\l = M(\sqrt{\beta_2}, \cdots, \sqrt{\beta_{n-1}})$. Note that $\l$ is chosen as in Lemma \ref{idea} for $T = \varnothing$ (see \eqref{choiceofq}). Therefore $\res_\l(c) \neq 0$, since $N \subset K_\l$ would mean $N \subset  M(\sqrt{\beta_2}, \cdots, \sqrt{\beta_{n-1}})$, which is a contradiction. Then the theorem follows from Proposition \ref{idea2}(iii). 
\end{proof}

\begin{thm}
\label{multiplicative}
Suppose that $E$ has multiplicative reduction at a prime $v_0$, where $v_0 \nmid 2$. Then there exist (infinitely many) $\chi \in \Xset(K)$ such that $r_2(E^\chi) = r_2(E) + 3$. If moreover, 
$E(K)[2] \cong \Z/2\Z$ and $v_0(\Delta_E)$ is odd where $v_0$ denote the normalized valuation of $K_{v_0}$, then there exist (infinitely many) $\chi \in \Xset(K)$ such that $r_2(E^\chi) = r_2(E) + 1$.
\end{thm}

\begin{proof}
If $E(K)[2] = 0$, \cite[Theorem 1.5]{MR} and Theorem \ref{inc2} prove the stronger statement that $A_E = \Z_{\ge 0}$. Suppose that $E(K)[2] \neq 0$. Choose the (non-trivial) quadratic unramified character $\psi_{v_0} \in \Xset(K_{v_0}).$  By local class field theory, $\psi_{v_0}(\Delta_E) = 1$ if and only if $v_0(\Delta_E)$ is even. By Corollary \ref{babo}, we have $|r_2(E) - r_2(E, \psi_{v_0})| \le 2$. Therefore, \cite[Proposition 1 and 2(a)]{kramer} and Theorem \ref{parity} show that $r_2(E) - r_2(E, \psi_{v_0})$ is either $-1$ or $1$. Let $T = \varnothing$ and choose $\l$ and $\chi$ as in Lemma \ref{idea}. If $\psi_{v_0}(\Delta_E) = 1$, \cite[Proposition 1 and 2(a)]{kramer} shows that $h_{v_0}(\psi_{v_0}) = 1$. Then Proposition \ref{idea2}(ii) and Theorem \ref{inc2} prove the first assertion. If $\psi_{v_0}(\Delta_E) = -1$ (so $E(K)[2] \cong \Z/2\Z$), \cite[Proposition 1 and 2(a)]{kramer} shows that $h_{v_0}(\psi_{v_0}) = 0$, so $\Sel_2(E) = \Sel_2(E, \psi_{v_0}).$ Therefore the second assertion follows from Proposition \ref{idea2}(i). 
\end{proof}

\section{An upper bound for $t_E$}
We continue to assume that $E$ is an elliptic curve over a number field $K$. Recall that $t_E$ is the smallest number in the set $A_E=\{r_2(E^\chi): \chi \in \Xset(K)\}$. In this section, we study $t_E$. Let $s_2$ be the number of complex places of $K$.

\begin{exa}
\label{example}
Let $E_{(m)}$ be the elliptic curve over $K$ defined by the equation
\begin{equation}
\label{klagsbrunex}
E_{(m)}: y^2 + xy = x^3 - 128m^2x^2 - 48m^2x - 4m^2.
\end{equation}
Suppose that $1+256m^2 \notin (K^\times)^2$. Then $E_{(m)}$ has a single point $(-1/4, 1/8)$ of order $2$ in $E_{(m)}(K)$. In \cite{klagsbrun}, Klagsbrun shows that $t_{E_{(m)}} \ge s_2 + 1$. Note that in this paper $r_2(E)$ is defined (slightly) differently from that defined in \cite{klagsbrun} (In \cite{klagsbrun}, the author subtracts the contribution of rational $2$-torsion points from $\dimtwo(\Sel_2(E))$ for the ``$2$-Selmer rank''). As his example suggests, $t_E$ can be a lot bigger than the trivial lower bound $\dimtwo(E(K)[2])$.
\end{exa}

\begin{rem}
\label{extendklagsbrunexample}
If $K$ contains $\sqrt{1+256m^2}$, then $E(K)$ contains all $2$-torsion points. In this case, we still can prove $t_{E_{(m)}} \ge s_2$ using the argument in \cite{klagsbrun}. Note that all Lemmas and Propositions in Section 3 in {\it op.~cit.}~ can be proved by the exactly same methods. However, in the proof of Proposition 4.1 in {\it op.~cit.}, now the map from $\Sel_{\phi}(E)$ to $\Sel_2(E)$ is injective and $\dimtwo(\Sel_{\hat{\phi}}(E'/K)) \ge 0$, so $r_2(E) \ge \text{ord}_2(\mathcal{T}(E/E'))$ is the correct lower bound we can get from applying the argument of the proof of Proposition 4.1 in {\it op.~cit.}.
\end{rem}

For the rest of the paper, we let $|\Sigma| = n$ and $E[2] \subset E(K)$. Note that this means if $v \notin \Sigma$, then $v \in \cP_2$. For a character $\chi \in \Xset(K)$ and a place $v$, we write $\chi_v \in \Xset(K_v)$ for the restriction of $\chi$ to $K_v^\times$ via the local reciprocity map. Let $L = K(\sqrt{\O_{K, \Sigma}^\times})$. Let $v_0 \in \Sigma$ and $\psi_{v_0} \in \Xset(K_{v_0})$. We discuss an upper bound for $t_E$ from now on.

\begin{defn}
\label{mrlemma}
If $\l \notin \Sigma$, the composition map
$$
\xymatrix@R=5pt@C=15pt{
\Sel_2(E, \psi_{v_0}) \ar^-{\res_\l}[r] & \Hom_{ur}(G_{K_\l}, E[2]) \cong E[2] 
}
$$
is given by sending $c \in \Sel_2(E, \psi_{v_0}) \subset \Hom(G_K, E[2])$ to $c(\Frob_\l)$, where $\Frob_\l$ is a Frobenius automorphism at $\l$ (note that $\res_\l(c) \neq 0$ if and only if $c(\Frob_\l) \neq 0$).
\end{defn}

\begin{lem}
\label{linearalgebrasurj}
Suppose that $\phi_1, \cdots ,\phi_n$ are homomorphisms from $\Ftwo^m$ to $\Ftwo^2$ where $m = n + k$ and $1 \le k \le n$ such that $\cap_{i=1}^n \ker(\phi_i) = \{0\}$. Then there exist $i_1, \cdots, i_k$ such that $\phi_{i_1} \times \cdots \times \phi_{i_k}: \Ftwo^m \to (\Ftwo^2)^k$ sending $v \in \Ftwo^m$ to $(\phi_{i_1}(v), \cdots, \phi_{i_k}(v))$ is surjective. 
\end{lem}

\begin{proof}
Define $s_j = \dimtwo(\im(\phi_{1} \times \cdots \times \phi_{j})).$ Then clearly $s_j = s_{j-1}$ or $s_{j} = s_{j-1} + 1$ or $s_{j} = s_{j-1} + 2$. Then there are at least $k$ many $j$ such that $s_{j} = s_{j-1} + 2$. Collect all $j$ such that $s_{j} = s_{j-1} + 2$ and name them $i_1 <  \cdots < i_k < \cdots$. Then it is easy to see $\phi_{i_1} \times \cdots \times \phi_{i_k}$ is surjective.
\end{proof}

\begin{prop}
\label{2n}
Let $v_0 \in \Sigma$ and $\psi_{v_0} \in \Xset(K_{v_0})$. Then $r_2(E, \psi_{v_0}) \le 2n$. 
\end{prop}

\begin{proof}
Clearly, we have $\Sel_2(E, \psi_{v_0}) \subseteq \Hom(G_K, E[2])$. For all nonzero $s \in \Sel_2(E, \psi_{v_0})$, we claim that $\overline{K}^{\ker(s)} \subseteq L = K(\sqrt{\O_{K,\Sigma}^\times})$. Indeed, for any quadratic extension $K(\sqrt{a})/K$, where all primes not in $\Sigma$ are unramified, one can replace $a$ with an element in $\O_{K,\Sigma}^\times$ because $\Pic(O_{K,\Sigma}) =1$. Now the claim follows easily once we note that $\overline{K}^{\ker(s)}$ is a compositum of (possibly the same) quadratic extensions, where all primes not in $\Sigma$ are unramified.
 Therefore, $\Sel_2(E, \psi_{v_0}) \subseteq \Hom(\Gal(L/K), E[2])$ by the Inflation-Restriction Sequence. By Lemma \ref{unitdim}, $\dimtwo(\Gal(L/K)) =n$, whence the result. 
\end{proof}

\begin{thm}
\label{upperbound}
Suppose $E[2] \subset E(K)$. If $r_2(E, \psi_{v_0}) = n + k$ for $2 \le k \le n$, then there exist $E^\chi$ such that $r_2(E^\chi) = n - k + 2$. In particular $t_E \le n+1$.  
\end{thm}

\begin{proof}
 Let $\beta_1,  \cdots, \beta_n$ be a basis of $\O_{K,\Sigma}^\times/(\O_{K,\Sigma}^\times)^2$. Let $L = K(\sqrt{\beta_1}, \cdots, \sqrt{\beta_n})$. Define $\sigma_i \in \Gal(L/K)$ so that $\sigma_i(\sqrt{\beta_i}) = -\sqrt{\beta_i}$ and $\sigma_i(\sqrt{\beta_j}) = \sqrt{\beta_j}$ for $j \neq i$. Note that an element $s \in \Sel_2(E, \psi_{v_0})$ is determined by $s(\sigma_1), \cdots, s(\sigma_n) \in E[2]$. Define $t_i \in \Hom(\Sel_2(E, \psi_{v_0}), E[2])$ sending $s \in \Sel_2(E, \psi_{v_0})$ to $s(\sigma_i)$. Applying Lemma \ref{linearalgebrasurj}, without loss of generality, we may assume $t_1 \times \cdots \times t_k$ is a surjection from $\Sel_2(E, \psi_{v_0})$ to $E[2]^k$. In other words, there exist $s_{2i-1}, s_{2i}$ for $0 \le i \le k$ such that 
\begin{itemize}
\item
$s_{2i-1}(\sigma_i) = P_1$ and $s_{2i}(\sigma_i) = P_2,$ where $P_1, P_2 \in E[2]$ is a basis of $E[2]$, 
\item
$s_{2i-1}(\sigma_j) = s_{2i}(\sigma_j) = 0$ for $1 \le j \neq i \le k$. 
\end{itemize}
For $ 1 \le i \le k$, let $\omega_i \in \cP_2$ be a prime such that $\Frob_{\omega_i} = \sigma_i$ in $\Gal(L/K)$.  Then by Definition \ref{mrlemma} we have that
\begin{enumerate}
\item
$\res_{\omega_i}(s_{2i-1})$ and $\res_{\omega_i}(s_{2i})$ generate $\alpha_{\omega_i}(1_{\omega_i}) = \Hom_{\ur}(G_{K_{\omega_i}}, E[2])$, and
\item
$\res_{\omega_j}(s_{2i-1}) = \res_{\omega_j}(s_{2i}) = 0$ for $ 1 \le j \neq i \le k$. 
\end{enumerate}
Let $T= \{\omega_1, \cdots, \omega_k\}.$ Let $\psi_i \in \Xset_{\ram}(K_{\omega_i})$. Then by Corollary \ref{ressur} and Theorem \ref{ptd}, we have $\Sel_2(E, \psi_{v_0}) = \Sel_2^T(E, \psi_{v_0})$ and
\begin{equation}
\label{good}
r_2(E, \psi_1, \cdots, \psi_k, \psi_{v_0}) = r_2(E, \psi_{v_0})- 2k.  
\end{equation}
By Lemma \ref{idea}, there exist $\l \in \cP_2 \backslash T$ and $\chi \in \Xset(K)$ so that 
\begin{itemize}
\item
$\chi_{v_0} = \psi_{v_0}$
\item
$\chi_{v} = 1_v$ for $v \in \Sigma \backslash \{v_0\}$,
\item
$\chi_\omega$ is ramified for $\omega \in T$, 
\item
$\chi_{\p}$ is unramified for $\p \notin \Sigma \cup T \cup \{\l\}$, and
\item
$\chi_\l$ is ramified.
\end{itemize}
Then $\Sel_2(E^\chi) = \Sel_2(E, \chi_{\omega_1}, \cdots, \chi_{\omega_k}, \psi_{v_0}, \chi_{\l})$. Theorem \ref{parity}, Lemma \ref{ramhv}, and Corollary \ref{babo} show 
$$
|r_2(E, \chi_{\omega_1}, \cdots, \chi_{\omega_k}, \psi_{v_0}, \chi_{\l})  - r_2(E, \chi_{\omega_1}, \cdots, \chi_{\omega_k}, \psi_{v_0})|
$$
is even and less than or equal to $2$, so by \eqref{good}, we have $r_2(E^\chi) = r_2(E, \psi_{v_0}) - 2k - 2$ or $r_2(E, \psi_{v_0}) - 2k$ or $r_2(E, \psi_{v_0}) - 2k + 2$. In any case, by Theorem \ref{inc2}, there exist infinitely many $\varphi \in \Xset(K)$ such that $r_2(E^{\varphi}/K) = r_2(E, \psi_{v_0}) - 2k + 2 = n + k - 2k + 2 < n + 1$. Proposition \ref{2n} with putting $\psi_{v_0} = 1_{v_0}$ shows that $t_E \le n+1$. 
\end{proof}

\begin{lem}
\label{finallemma}
Suppose that there exist $c_1, c_2 \in \Sel_2(E, \psi_{v_0})$ such that $\res_{\omega}(c_1)$ and $\res_{\omega}(c_2)$ generate $\alpha_{\omega}(1_{\omega}) = \Hom_{\ur}(G_{\omega}, E[2])$ for some prime $\omega \notin \Sigma.$ Then there exist infinitely many $\varphi \in \Xset(K)$ such that $r_2(E^\varphi/K) = r_2(E, \psi_{v_0})$. 
\end{lem}

\begin{proof}
Let $T= \{\omega\}$. By Lemma \ref{idea}, there exist infinitely many $\l \notin \Sigma \cup T$ for which there exists a character $\chi \in \Xset(K)$ such that
\begin{itemize}
\item
$\chi_{v_0} = \psi_{v_0}$,
\item
$\chi_{v} = 1_v$ for $v \in \Sigma \backslash \{v_0\}$,
\item
$\chi_{\omega}$ is ramified,
\item
$\chi_{\p}$ is unramified for all $\p \notin \Sigma \cup T \cup \{\l\}$, and
\item
$\chi_\l$ is ramified.
\end{itemize}
Note that $\Sel_2(E^\chi) = \Sel_2(E, \psi_{v_0}, \chi_{\omega}, \chi_\l)$ by Lemma \ref{localzero}.
Let $S = \{\omega, \l \}$. Then $r_2(E, \psi_{v_0}) \ge \dimtwo(\Sel_{2, S}(E^\chi)) + 2$, since by the condition on $c_1, c_2, \omega$ the following map is surjective 
$$
\res_{\omega} : \Sel_2(E, \psi_{v_0})/\Sel_{2, S}(E^\chi) \to \Hom_{\ur}(G_{K_{\omega}}, E[2]).
$$
Note that $c_1, c_2, c_1+c_2 \in \Sel_2^S(E^\chi) \backslash \Sel_2(E^\chi)$ since $\alpha_{\omega}(1_{\omega}) \cap \alpha_{\omega}(\chi_{\omega}) = \{0\}$ (Lemma \ref{ramhv}). Therefore $\dimtwo(\Sel_2^S(E^\chi)) \ge r_2(E^\chi) + 2$.
Theorem \ref{ptd} shows that $\dimtwo(\Sel_2^S(E^\chi))- \dimtwo(\Sel_{2, S}(E^\chi)) = 4$. Then it follows that $r_2(E, \psi_{v_0}) \ge r_2(E^\chi)$ and $r_2(E, \psi_{v_0}) \equiv r_2(E^\chi) \text{ (mod $2$)}$ by Theorem \ref{parity}. Then the assertion follows from Theorem \ref{inc2}. 
\end{proof}

\begin{thm}
\label{upperbound2}
If $E$ does not satisfy the constant $2$-Selmer parity condition, then $t_E \le n$.
\end{thm}

\begin{proof}
If $r_2(E) \equiv n$ (mod $2$), the result follows from Theorem \ref{upperbound}. From now on, we assume that $r_2(E) \not\equiv n$ (mod $2$). Since $E$ does not satisfy the constant $2$-Selmer parity, there exist $v_0 \in \Sigma$ and $\psi_{v_0} \in \Xset(K_{v_0})$ such that $r_2(E, \psi_{v_0}) \equiv n$ (mod $2$) by Theorem \ref{parity} (note that since $E[2] \subset E(K)$, all primes outside $\Sigma$ are in $\cP_2$, so twisting locally at primes not in $\Sigma$ does not change the parity by Lemma \ref{localzero} and Lemma \ref{ramhv}). If $r_2(E, \psi_{v_0}) \le n-2$ or $r_2(E, \psi_{v_0}) \ge n+2$, then the result follows from Proposition \ref{idea2}(ii), Theorem \ref{upperbound}, respectively. Let $r_2(E, \psi_{v_0}) = n$. If $\psi_{v_0}(\O_{K, \Sigma}^\times) = 1$, Proposition \ref{lastcase} shows the result. Now let $\beta_1, \cdots, \beta_n$ be a basis of $\O_{K,\Sigma}^\times/(\O_{K,\Sigma}^\times)^2$ such that $\psi_{v_0}(\beta_1) = -1$ and $\psi_{v_0}(\beta_2) = \psi_{v_0}(\beta_3) = \cdots = \psi_{v_0}(\beta_n) = 1$. 

Define $\sigma_1, \cdots, \sigma_n$ and $t_1, \cdots, t_n$ as in the proof of Theorem \ref{upperbound}. If $\dimtwo(\im(t_1)) \ge 1$, let $c \in \Sel_2(E, \psi_{v_0})$ and $c(\sigma_1) \neq 0$. Choose $\l$ ($T = \varnothing$) as in Lemma \ref{idea}, i.e., $\Frob_\l = \sigma_1$ in $L/K$ (see \eqref{choiceofq}). Then $c(\Frob_\l) = c(\sigma_1) \neq 0$, so Definition \ref{mrlemma} shows $\res_\l(c) \neq 0$. Then the result follows from Proposition \ref{idea2}(iii). Therefore for the rest of the proof, assume that $\dimtwo(\im(t_1)) = 0$. Then without loss of generality, we may assume $\dimtwo(\im(t_2)) = 2$. Choose $\omega \notin \Sigma$ so that $\Frob_{\omega} = \sigma_2$ in $\Gal(L/K)$. 
Then Definition \ref{mrlemma} shows that there exist $c_1, c_2 \in \Sel_2(E, \psi_{v_0})$ such that $\res_{\omega}(c_1)$ and $\res_{\omega}(c_2)$ generate $\Hom_\ur(G_{K_{\omega}}, E[2])$. Now Lemma \ref{finallemma} completes the proof. 
\end{proof}

\section*{Acknowledgements}
The author is very grateful to his advisor Karl Rubin for helpful suggestions and discussion. The author also thanks to the referee for careful reading the manuscript and many comments.

\bibliographystyle{plain}
\bibliography{secondmjrefe}

\end{document}